\documentclass[12pt]{amsart}

\newtheorem{theorem}{Theorem}[section]
\newtheorem{lemma}[theorem]{Lemma}
\newtheorem{corollary}[theorem]{Corollary}

\theoremstyle{plain}
\newtheorem{definition}[theorem]{Definition}

\newtheorem{question}[theorem]{Question}

\theoremstyle{definition}

\theoremstyle{remark}

\numberwithin{equation}{section}

\usepackage{CJK}
\usepackage{fancyhdr}
\usepackage{amscd}
\usepackage{amsmath}
\usepackage{amsthm}
\usepackage{amssymb}
\begin{document}

\title[Soul\'e's variant of Bloch-Quillen identification]{On extending Soul\'e's variant of Bloch-Quillen identification}

\author{Sen Yang}
\address{Yau Mathematical Sciences Center \\
Tsinghua University \\
Beijing, China\\
}
\email{syang@math.tsinghua.edu.cn; senyangmath@gmail.com}

\date{}

\maketitle

\begin{abstract}
Based on Balmer's tensor triangular Chow group \cite{B-5}, we propose (Milnor)K-theoretic Chow groups of derived categories of schemes. These Milnor K-theoretic Chow groups recover the classical ones \cite{F} for smooth projective varieties and can detect nilpotent, while the classical ones can't do. 

As an application, we extend Soul\'e's variant of Bloch-Quillen identification from smooth projective varieties to their trivial infinitesimal thickenings. This answers affirmatively a question by Green-Griffiths for trivial deformations, see Question ~\ref{question: Green-Griffiths} below.
  
\end{abstract}

\tableofcontents

\section{Introduction}
\label{Introduction}
Let $f: \mathcal{X} \to S$ be a smooth projective
morphism, where $S = \mathrm{Spec}(k[[t]])$ and $k$ is a field of characteristic $0$. Let $X_{j}= \mathcal{X} \times_{S} S_{j}$, where $S_{j} = \mathrm{Spec}(k[t]/ t^{j+1})$.   
We use $X$ to denote $X_{0}$, and call the family $\{ X_{j}\}_{j} $ a deformation of $X$, where $X_{j}$ is called the $j$-th infinitesimal thickening of $X$. In particular, the family $\{ X_{j}\}_{j} $ is a trivial deformation of $X$ and $X_{j}$ is called the $j$-th trivial infinitesimal thickening of $X$, if, for each $j$,  $X_{j}= X \times_{k} S_{j}$. 

In \cite{GG-1, GG-2}, Green-Griffiths studies infinitesimal deformations of  Chow groups. Fundamental to their work is the Soul\'e's  variant of the Bloch-Quillen identification
\[
 CH^{q}(X)_{\mathbb{Q}}= H^{q}(X, K^{M}_{q}(O_{X}))_{\mathbb{Q}},
\]
where $K^{M}_{q}(O_{X})$ is the 
Milnor K-theory sheaf associated to the presheaf
\[
  U \to K^{M}_{q}(O_{X}(U)).
\]

On page 471 of \cite{GG-1}, Green-Griffiths suggested that it would be interesting to extend Bloch-Quillen
 identification to infinitesimal thickening $X_{j}$:
\begin{question} \cite{GG-1} \label{question: Green-Griffiths}
Let $X$ denote the closed fiber $X_{0}$ and $X_{j}$ denote the $j$-th infinitesimal thickening of $X$(not necessarily trivial)as above, do we have the following identification
{\small
\[
 CH^{q}(X_{j})_{\mathbb{Q}} = H^{q}(X_{j}, K^{M}_{q}(O_{X_{j}}))_{\mathbb{Q}} \ ?
\]
}
where $K^{M}_{q}(O_{X_{j}})$ is the 
Milnor K-theory sheaf associated to the presheaf
\[
  U \to K^{M}_{q}(O_{X_{j}}(U)).
\]

\end{question}

This question inspires us to propose a new definition of Chow groups capturing the nilpotent which is useful for studying deformation problems. Our starting point is to look at the derived category $D^{\mathrm{perf}}(X)$ obtained from the exact category of perfect complexes of $O_{X}$-modules. 

In \cite{B-5}, Balmer introduces tensor triangular Chow groups.  The new insight is to allow the coefficients of $q$-cycles to lie in Grothendieck groups of suitable triangulated categories.  Balmer's idea is followed by S. ~Klein \cite{K} and the author \cite{Y-2}.

In Section ~\ref{K-theoretic Chow groups of derived categoies}, we recall Balmer's tensor triangular Chow groups and propose K-theoretic Chow groups of derived categoies of schemes by slightly modifying Balmer's. 
In Section ~\ref{Definition}, we propose Milnor K-theoretic Chow groups with an additional assumption. We compute relative K-groups with support in Section ~\ref{Goodwillie isomorphism} and prove the main result in Section ~\ref{Bloch's formula}, which answers Question ~\ref{question: Green-Griffiths} for trivial infinitesimal thickening, i.e., for each $j$,  $X_{j}= X \times_{k} S_{j}$.

\textbf{Acknowledgements}
This note is a revision of \cite{Y-2}.The author is very grateful to Mark Green and Phillip Griffiths for asking interesting questions and for telling him their paper \cite{GG-1} which started \cite{Y-2}. He also thanks the following professors for discussions and/or comments on \cite{Y-2}: Paul Balmer,  Ben Dribus, Jerome Hoffman, Sebastian Klein, Claudio Pedrini, Marco Schlichting and Burt Totaro. Part of results has been done in \cite{Y-1}.

Many thanks to the anonymous referee(s) for careful reading and professional suggestions that improved 
this note a lot.  

\textbf{Conventions} 
K-theory used in this note will be Thomason-Trobaugh K-theory \cite{TT}, if not stated otherwise.  For any abelian group $M$, $M_{\mathbb{Q}}$ denotes the image of $M$ in $M \otimes_{\mathbb{Z}} \mathbb{Q}$.   ``Dimension" always means Krull dimension, while Balmer uses general dimension functions in \cite{B-5}.

\section{K-theoretic Chow groups of derived categories}
\label{K-theoretic Chow groups of derived categoies}
In this section, $X$ is an equidimensional noetherian scheme over a field of finite Krull dimension $d$. As explained in \cite{B-4}, one can filter the tensor triangulated  category $\mathcal{L} = D^{\mathrm{perf}}(X)$ by dimension of support
\[
\dots \subset \mathcal{L}_{(p)}(X) \subset \mathcal{L}_{(p+1)}(X) \subset \dots \subset \mathcal{L}, 
\]  
where $\mathcal{L}_{(p)}(X)$ is defined to be
\[
  \mathcal{L}_{(p)}(X) := \{ E \in D^{\mathrm{perf}}(X) \mid \mathrm{codim_{Krull}}(\mathrm{supph}(E)) \geq -p \},
\]
where the closed subset $\mathrm{supph}(E) \subset X$ is the support of the total homology of the perfect complex $E$.

Let $(\mathcal{L}_{(p)}(X)/\mathcal{L}_{(p-1)}(X))^{\#}$ denote the idempotent completion of the Verdier quotient $\mathcal{L}_{(p)}(X)/\mathcal{L}_{(p-1)}(X)$.
\begin{theorem} \cite{B-4}  \label{theorem: BalmerTheorem}
Localization induces an equivalence
\[
 (\mathcal{L}_{(p)}(X)/\mathcal{L}_{(p-1)}(X))^{\#}  \simeq \bigsqcup_{x \in X^{(-p)}}D_{{x}}^{\mathrm{perf}}(X)
\]
between the idempotent completion of the quotient $\mathcal{L}_{(p)}(X)/\mathcal{L}_{(p-1)}(X)$ and the coproduct over $x \in X^{(-p)}$ of the derived category of perfect complexes of $ O_{X,x}$-modules with homology supported on the closed point $x \in \mathrm{Spec}(O_{X,x})$.
\end{theorem} 

 The short sequence 
\[
 \mathcal{L}_{(p-1)}(X) \to \mathcal{L}_{(p)}(X) \to (\mathcal{L}_{(p)}(X)/\mathcal{L}_{(p-1)}(X))^{\#},
\]
which is exact up to summand,  induces the following homotopy fibration of K-theory spectrum:
\[
 \mathcal{K}(\mathcal{L}_{(p-1)}(X))  \to \mathcal{K}(\mathcal{L}_{(p)}(X)) \to \mathcal{K}((\mathcal{L}_{(p)}(X)/\mathcal{L}_{(p-1)}(X))^{\#}).
\]
 
 As pointed out in \cite{B-4}, this fibration gives rise to a long exact sequence:
{\scriptsize
\[
 \dots \to K_{n}(\mathcal{L}_{(p-1)}(X)) \rightarrow  K_{n}(\mathcal{L}_{(p)}(X)) \xrightarrow{i} K_{n}((\mathcal{L}_{p}(X)/\mathcal{L}_{p-1}(X))^{\#}) \xrightarrow{k}  K_{n-1}(\mathcal{L}_{(p-1)}(X)) \to \dots,
\]
 }
 which produces an exact couple as usual and then gives rise to the associated coniveau spectral sequence with $E_{1}$-term:
\[
 E_{1}^{p,q} = K_{-p-q}((\mathcal{L}_{(-p)}(X)/\mathcal{L}_{(-p-1)}(X))^{\#}),
\]
the differential $d_{1,X}^{p,q}$ is the composition $i \circ k$ as usual
{\scriptsize
\begin{align*}
d_{1,X}^{p,q}: & K_{-p-q}((\mathcal{L}_{(-p)}(X)/\mathcal{L}_{(-p-1)}(X))^{\#}) \xrightarrow{k} K_{-p-q-1}(\mathcal{L}_{(-p-1)}(X)) \\
     & \xrightarrow{i}  K_{-p-q-1}((\mathcal{L}_{(-p-1)}(X)/\mathcal{L}_{(-p-2)}(X))^{\#}).
\end{align*}
}

Balmer's definition of tensor triangular Chow groups applies to the tensor triangulated category $\mathcal{L} = D^{\mathrm{perf}}(X)$:
\begin{definition} \cite{B-5} \label{definition: Balmercycles}
Let $q \in \mathbb{Z}$, one defines K-theoretic $q$-cycles associated to the tensor triangulated category $\mathcal{L} = D^{\mathrm{perf}}(X)$ to be
\[
 Z_{q}(\mathcal{L}) := K_{0}((\mathcal{L}_{(q)}/\mathcal{L}_{(q-1)})^{\#})= \bigoplus_{x \in X^{(-q)}}K_{0}(O_{X,x} \ \mathrm{on} \ x),
\] 
where $K_{0}$ is the Grothendieck K-group(the quotient of the monoid of isomorphism class $[a]$ of objects under $\oplus$, by the submonoid of those $[a] + [\sum b] + [c]$ for which there exists a distinguish triangle $ a \to b \to c \to \sum a$). 
\end{definition}

A K-theoretic $q$-cycles can be written as $\sum \limits_{x \in X^{(-q)}}\lambda_{x}\cdot \overline{\{ x \}}$, for $ \sum \limits_{x \in X^{(-q)}}\lambda_{x} \in \bigoplus \limits_{x \in X^{(-q)}}K_{0}(O_{X,x} \ \mathrm{on} \ x) $. Balmer's  new insight is to allow coefficients $\lambda_{x}$ to live in the Grothendieck groups, not $\mathbb{Z}$. 

\begin{definition} \cite{B-5} \label{definition: Balmerboundaries}
Let $q \in \mathbb{Z}$, we use  $\mathrm{Ker}(j)$ to denote the Kernel of $K_{0}(\mathcal{L}_{(q)}) \xrightarrow{j} K_{0}(\mathcal{L}_{(q+1)})$. The K-theoretic $q$-boundaries $B_{q}(\mathcal{L})$ is defined as the image of $\mathrm{Ker}(j)$ in $Z_{q}(\mathcal{L})$
 \[
   B_{q}(\mathcal{L}):= i \circ \mathrm{Ker}(j),
 \]
where $i: K_{0}(\mathcal{L}_{(q)}) \to K_{0}(\mathcal{L}_{(q)}/\mathcal{L}_{(q-1)})^{\#})(=Z_{q}(\mathcal{L}))$.

The K-theoretic Chow group of $q$-cycles in $\mathcal{L}$, denoted $CH_{p}(\mathcal{L})$, is defined 
to be the quotient of $q$-cycles by $q$-boundaries
\[
  CH_{q}(\mathcal{L}) := \dfrac{Z_{q}(\mathcal{L})}{B_{q}(\mathcal{L})}.
\]
\end{definition}

\begin{definition}  \cite{B-4}  \label{definition:Gersten}
For $X$ an equidimensional noetherian scheme over a field of finite Krull dimension $d$ and for each integer $q$ satisfying $1 \leq q \leq d+1$, the $q$-th $\mathrm{Gersten \ complex}$ $G_{q}$ is defined to be the $(-q)$-th line of $E_{1}$-page of the above coniveau spectral sequence
{\footnotesize
\begin{align*}
 G_{q}: \ & 0  \rightarrow  \bigoplus_{x \in X^{(0)}}K_{q}(O_{X,x}) \rightarrow \dots \rightarrow \bigoplus_{x \in X^{(q-1)}}K_{1}(O_{X,x} \ \mathrm{on} \ x) \\
   & \xrightarrow{d_{1,X}^{q-1,-q}} \bigoplus_{x \in X^{(q)}}K_{0}(O_{X,x} \ \mathrm{on} \ x)  \xrightarrow{d_{1,X}^{q,-q}} \bigoplus_{x \in X^{(q+1)}}K_{-1}(O_{X,x} \ \mathrm{on} \ x)   \rightarrow \cdots\\ 
   & \rightarrow \bigoplus_{x \in X^{(d)}}K_{q-d}(O_{X,x} \ \mathrm{on} \ x) \rightarrow 0.
\end{align*}
}

The $q$-th $\mathrm{ augmented \ Gersten \ complex}$, still denoted $G_{q}$ by abuse of notations, is defined to be
the following complex, 
{\footnotesize
\begin{align*}
 G_{q}: \ & 0 \to K_{q}(X) \to  \bigoplus_{x \in X^{(0)}}K_{q}(O_{X,x}) \to \dots \to \bigoplus_{x \in X^{(q-1)}}K_{1}(O_{X,x} \ \mathrm{on} \ x) \\
   & \xrightarrow{d_{1,X}^{q-1,-q}} \bigoplus_{x \in X^{(q)}}K_{0}(O_{X,x} \ \mathrm{on} \ x)  \xrightarrow{d_{1,X}^{q,-q}} \bigoplus_{x \in X^{(q+1)}}K_{-1}(O_{X,x} \ \mathrm{on} \ x)   \to \cdots\\ 
   & \to \bigoplus_{x \in X^{(d)}}K_{q-d}(O_{X,x} \ \mathrm{on} \ x) \to 0.
\end{align*}
}
\end{definition}

\begin{theorem}
Balmer's K-theoretic $(-q)$-boundaries $B_{-q}(\mathcal{L})$ for $\mathcal{L}=D^{\mathrm{perf}}(X)$ agrees with  the image of the differential $d_{1,X}^{q-1,-q}$:
\[
 B_{-q}(\mathcal{L}) = \mathrm{Im}(d_{1,X}^{q-1,-q}).
\]
\end{theorem}

\begin{proof}
The long exact sequence
\[
 \dots \to K_{1}((\mathcal{L}_{(-q+1)}/\mathcal{L}_{(-q)})^{\#}) \xrightarrow{k}  K_{0}(\mathcal{L}_{(-q)}) \xrightarrow{j}  K_{0}(\mathcal{L}_{(-q+1)}) \to \dots
\]
shows that $\mathrm{Ker}(j) = \mathrm{Im}(k)$. So $B_{-q}(\mathcal{L})= i \circ \mathrm{Ker}(j) = i \circ \mathrm{Im}(k)$,
where $i: K_{0}(\mathcal{L}_{(-q)}) \to K_{0}((\mathcal{L}_{(-q)}/\mathcal{L}_{(-q-1)})^{\#})$.

The conclusion follows, since the differential $d_{1,X}^{q-1,-q}$ is the composition $d = i \circ k$ 
\[
d_{1,X}^{q-1,-q}: K_{1}((\mathcal{L}_{(-q+1)}/\mathcal{L}_{(-q)})^{\#}) \xrightarrow{k} K_{0}(\mathcal{L}_{(-q)}) \xrightarrow{i}  K_{0}((\mathcal{L}_{(-q)}/\mathcal{L}_{(-q-1)})^{\#}).
\]
\end{proof}

As pointed out in \cite{B-4}, taking idempotent completion can result in the appearance of negative K-groups. 
In order to include this important information into our study, we propose our K-theoretic Chow groups of $D^{\mathrm{perf}}(X)$ by slightly modifying Balmer's as follows: 
\begin{definition} \label{definition: K-theoretic Chowgp}
The K-theoretic $q$-cycles and K-theoretic rational equivalence of $(X,O_{X})$, denoted $Z_{q}(D^{\mathrm{perf}}(X))$ and $Z_{q,\mathrm{rat}}(D^{\mathrm{perf}}(X))$ respectively, are defined to be 
\[
  Z_{q}(D^{\mathrm{perf}}(X)):= \mathrm{Ker}(d_{1,X}^{q,-q}), Z_{q,\mathrm{rat}}(D^{\mathrm{perf}}(X)):=\mathrm{Im}(d_{1,X}^{q-1,-q}).
\]

The $q$-th K-theoretic Chow group  of $(X,O_{X})$, denoted by $CH_{q}(D^{\mathrm{perf}}(X))$, is defined to be
\[
  CH_{q}(D^{\mathrm{perf}}(X)):= \dfrac{Z_{q}(D^{\mathrm{perf}}(X))}{Z_{q,\mathrm{rat}}(D^{\mathrm{perf}}(X))}.
\]
\end{definition}

It is clear that our K-theoretic Chow groups are cohomology groups of the Gersten complex and that these K-theoretic Chow groups are subgroups of Balmer's:
\begin{corollary} \label{corollary: compareBalmer}
Let $\mathcal{L} = D^{\mathrm{perf}}(X)$, we have the following:
\[
  Z_{q}(D^{\mathrm{perf}}(X)) \subseteq Z_{-q}(\mathcal{L}),
\]
\[
  Z_{q,\mathrm{rat}}(D^{\mathrm{perf}}(X)) = B_{-q}(\mathcal{L}),
\]
\[
  CH_{q}(D^{\mathrm{perf}}(X)) \subseteq CH_{-q}(\mathcal{L}).
\]
\end{corollary}

\section{Milnor K-theoretic Chow groups of derived categories}
\label{Milnor K-theoretic Chow groups of derived categories}

\subsection{Definition}
\label{Definition}
Let $X$ be an equidimensional noetherian scheme of finite Krull dimension $d$ over a field, and $Y$ be a closed subset of $X$, by abuse of notations, we use $K^{Q}_{m}( X \ \mathrm{on} \ Y)$ to denote Quillen K-groups with supports. In \cite{S}, Soul\'e showed that there exists Adams operations $\psi^{k}$ acting on Quillen K-groups with supports $K^{Q}_{m}( X \ \mathrm{on} \ Y)$, $m \geq 0$.  According to Weibel \cite{W-3}, we can extend Adams operations  $\psi^{k}$ to negative range by descending induction.

For every integer $m \geq 0$,  we have the following Bass fundamental exact sequence:
\begin{align*}
 0 \to & K^{Q}_{m}(X \ \mathrm{on} \ Y) \to K^{Q}_{m}(X[t] \ \mathrm{on} \ Y[t]) \oplus K^{Q}_{m}(X[t^{-1}] \ \mathrm{on} \ Y[t^{-1}])  \\
 & \to K^{Q}_{m}(X[t,t^{-1}] \ \mathrm{on} \ Y[t,t^{-1}]) \to  K^{Q}_{m-1}(X \ \mathrm{on} \ Y)  \to 0.
\end{align*}

For any $x \in K^{Q}_{-1}(X \ \mathrm{on} \ Y)$, we have $x\cdot t \in K^{Q}_{0}(X[t,t^{-1}] \ \mathrm{on} \ Y[t,t^{-1}])$, where $ t \in K_{1}(k[t,t^{-1}])$. We have
\[
 \psi^{k}(x\cdot t ) = \psi^{k}(x)\psi^{k}(t)= \psi^{k}(x) k\cdot t.
\]

Tensoring with $\mathbb{Q}$, we have obtained Adams operations $\psi^{k}$ on $K^{Q}_{-1}(X \ \mathrm{on} \ Y)$:
\[
 \psi^{k}(x)= \dfrac{\psi^{k}(x\cdot t )}{k\cdot t}.
\]
Continuing this procedure, we obtain Adams operations on $K^{Q}_{m}( X \ \mathrm{on} \ Y)$, where $m \in \mathbb{Z}$ .

We note that the Thomason-Trobaugh K-group with support $K_{m}(O_{X,x} \ \mathrm{on} \ x)$, appearing in Definition ~\ref{definition:Gersten}, is isomorphic to Quillen K-group with support $K^{Q}_{m}(O_{X,x} \ \mathrm{on} \ x)$. So Adams operations $\psi^{k}$ exist on $K_{m}(O_{X,x} \ \mathrm{on} \ x)$. Keeping Soul\'e's variant of Bloch's formula in mind, we would like to define Milnor K-theoretic Chow groups of derived categories of schemes. However, we don't have Milnor K-group with support $K_{m}^{M}(O_{X,x} \ \mathrm{on} \ x)$ directly. The following theorem of Soul\'e \cite{S} suggests that we can use suitable eigenspace of Adams operations to rationally replace Milnor K-group.


\begin{theorem} [Theorem 5 in page 526 of \cite{S}] \label{theorem: Soule}
Let $X$ be a regular scheme of finite type over a field, with Krull dimension $d$,
$X^{(p)}$ the set of points of $X$ of codimension $p$. For $x \in X^{(p)}$, $k(x)$ is the residue field.
There exists the following isomorphism
\[
 K^{M}_{m}(k(x)) \xrightarrow{\cong} K^{(m)}_{m}(k(x))  \ \mathrm{modulo \ torsion},
\]
where $K_{m}^{(m)}$ is the eigenspace of $\psi^{k}=k^{m}$ and $\psi^{k}$ is Adams operations.
\end{theorem}

Following Soul\'e 's theorem, we define Milnor K-group with support $K_{m}^{M}(O_{X,x} \ \mathrm{on} \ x)$ to be 
suitable eigenspace of $K_{m}(O_{X,x} \ \mathrm{on} \ x)$:
\begin{definition}
Let $X$ be a $d$-equidimensional noetherian scheme and $x \in X^{(j)}$. For $m \in \mathbb{Z}$, Milnor K-group with support $K_{m}^{M}(O_{X,x} \ \mathrm{on} \ x)$ is rationally defined to be 
\[
  K_{m}^{M}(O_{X,x} \ \mathrm{on} \ x) := K_{m}^{(m+j)}(O_{X,x} \ \mathrm{on} \ x)_{\mathbb{Q}},
\] 
where $K_{m}^{(m+j)}$ is the eigenspace of $\psi^{k}=k^{m+j}$.
\end{definition}

The reason why we choose $K_{m}^{(m+j)}$ to define $K_{m}^{M}$ is inspired by another result of Soul\'e. 

Let $X$ be a regular scheme of finite type over a field, with Krull dimension $d$.  
For $x \in X^{(p)}$, $k(x)$ denotes the residue field. Let $f: \mathrm{Spec}(k(x)) \to \mathrm{Spec}(O_{X,x})$ denote the closed immersion. It is known that 
\[
f_{\ast}: K_{m}(k(x)) \xrightarrow{\cong} K_{m}(O_{X,x} \ \mathrm{on} \ x).
\]
\begin{lemma}  \cite{S}  \label{theorem: SouleRRWithoutdenominator}
Let $X$ be a regular scheme of finite type over a field, with Krull dimension $d$.
For $x \in X^{(p)}$, let $f: \mathrm{Spec}(k(x)) \to \mathrm{Spec}(O_{X,x})$ be the closed immersion.
 We have(for any integer $m$ and $i$)
 \[
f_{\ast}: K^{(i)}_{m}(k(x)) \xrightarrow{\cong} K^{(p+i)}_{m}(O_{X,x} \ \mathrm{on} \ x), \ \mathrm{modulo \ torsion}.
\]
\end{lemma}

\begin{proof}
 We sketch the proof very briefly  as follows.
Given $a \in K^{(i)}_{m}(k(x))$, one has
\[
\psi^{k}(f_{\ast}(a)) = f_{\ast}(\psi^{k}(p,a)) = f_{\ast}(k^{p}\psi^{k}(a))= f_{\ast}(k^{p}k^{i}a)= k^{p+i}f_{\ast}(a),
\]
The first identity is from Theorem 3 of  \cite{S}(page 517), the second one uses the formula $\psi^{k}(p,a)= k^{p}\psi^{k}(a)$, see line 29 of page 522 of \cite{S}.
 Hence, $f_{\ast}(a) \in K^{(p+i)}_{m}(O_{X,x} \ \mathrm{on} \ x)$. 

See line 16- 30 of page 522 and line 8-11 of page 527 of \cite{S} for related discussion if necessary.
\end{proof}

This lemma says that our definition of Milnor K-group with support is a honest generalization of the classical one, at least for regular case. 

Next, we would like to define Milnor K-theoretic Chow groups of derived categories of schemes by mimicking Definition ~\ref{definition: K-theoretic Chowgp}. In order to do that, we need to detect whether the differentials of the Gersten complex respect Adams operations. 

 If the differentials $d_{1, X}^{p,-q}$ of the Gersten complex in Definition ~\ref{definition:Gersten} respect Adams' operations, for every $i \in \mathbb{Z}$, then there exists the following finer complex  
{\footnotesize
\begin{align*}
 0 \to & \bigoplus_{x \in X^{(0)}}K_{q}^{(i)}(O_{X,x}) \to \dots \to \bigoplus_{x \in X^{(q-1)}}K_{1}^{(i)}(O_{X,x} \ \mathrm{on} \ x) 
   \xrightarrow{d_{1,X}^{q-1,-q}} \bigoplus_{x \in X^{(i)}}K_{0}^{(i)}(O_{X,x} \ \mathrm{on} \ x) \\  &\xrightarrow{d_{1,X}^{q,-q}} \bigoplus_{x \in X^{(q+1)}}K_{-1}^{(i)}(O_{X,x} \ \mathrm{on} \ x) \to \dots \to \bigoplus_{x \in X^{(d)}}K_{q-d}^{(i)}(O_{X,x} \ \mathrm{on} \ x) \to 0.
\end{align*}
}

In particular, we obtain the following refiner complex by taking $i=q$:
{\footnotesize
\begin{align*}
 0 \to & \bigoplus_{x \in X^{(0)}}K_{q}^{(q)}(O_{X,x})  \to \dots \to \bigoplus_{x \in X^{(q-1)}}K_{1}^{(q)}(O_{X,x} \ \mathrm{on} \ x) 
   \xrightarrow{d_{1,X}^{q-1,-q}} \bigoplus_{x \in X^{(q)}}K_{0}^{(q)}(O_{X,x} \ \mathrm{on} \ x) \\  
   &\xrightarrow{d_{1,X}^{q,-q}} \bigoplus_{x \in X^{(q+1)}}K_{-1}^{(q)}(O_{X,x} \ \mathrm{on} \ x) \to 
 \dots \to \bigoplus_{x \in X^{(d)}}K_{q-d}^{(q)}(O_{X,x} \ \mathrm{on} \ x) \to 0.
\end{align*}
}

Tensoring each term with $\mathbb{Q}$, this complex can be written as 
{\footnotesize
\begin{align}
 0 \to & \bigoplus_{x \in X^{(0)}}K_{q}^{M}(O_{X,x})  \to \dots \to \bigoplus_{x \in X^{(q-1)}}K_{1}^{M}(O_{X,x} \ \mathrm{on} \ x) 
   \xrightarrow{d_{1,X}^{q-1,-q}} \bigoplus_{x \in X^{(q)}}K_{0}^{M}(O_{X,x} \ \mathrm{on} \ x) \\ 
   &\xrightarrow{d_{1,X}^{q,-q}} \bigoplus_{x \in X^{(q+1)}}K_{-1}^{M}(O_{X,x} \ \mathrm{on} \ x) \to \dots
   \dots \to \bigoplus_{x \in X^{(d)}}K_{q-d}^{M}(O_{X,x} \ \mathrm{on} \ x) \to 0. \nonumber
\end{align} 
}

\begin{definition} \label{definition: MilnorKChow}
For $X$ an equidimensional noetherian scheme over a field of finite Krull dimension $d$, \textbf{if} the differentials $d_{1,X}^{p,-q}$ of the Gersten complex respect Adams operations, then 
the Milnor K-theoretic $q$-cycles and Milnor K-theoretic rational equivalence of $(X, O_{X})$, denoted $Z^{M}_{q}(D^{\mathrm{perf}}(X))$ and $Z^{M}_{q,\mathrm{rat}}(D^{\mathrm{perf}}(X))$ respectively, are defined to be 
\[
  Z^{M}_{q}(D^{\mathrm{perf}}(X)):= \mathrm{Ker}(d_{1,X}^{q,-q}), Z^{M}_{q,\mathrm{rat}}(D^{\mathrm{perf}}(X)):=\mathrm{Im}(d_{1,X}^{q-1,-q}), 
\]
where $d_{1,X}^{q-1,-q}$ and $d_{1,X}^{q,-q}$ are the differentials of the complex $\mathrm{(3.1)}$.

The $q$-th Milnor K-theoretic Chow group  of $(X,O_{X})$, denoted by $CH^{M}_{q}(D^{\mathrm{perf}}(X))$, is defined to be 
\[
  CH^{M}_{q}(D^{\mathrm{perf}}(X)):= \dfrac{Z^{M}_{q}(D^{\mathrm{perf}}(X))}{Z^{M}_{q,\mathrm{rat}}(D^{\mathrm{perf}}(X))}.
\]
\end{definition}

In the following, we show that this definition makes sense for smooth projective varieties and their trivial infinitesimal  thickenings.

\subsection{Goodwillie  isomorphism}
\label{Goodwillie isomorphism}
Let's recall that in \cite{G-1} Goodwillie shows the relative 
Chern character is an isomorphism between the relative K-group $K_{n}(A,I)$ and negative cyclic homology $HN_{n}(A,I)$, where $A$ is a commutative $\mathbb{Q}$-algebra and $I$
is a nilpotent ideal in $A$. This is further studied by Cathelineau. 
\begin{theorem} \cite{C,G-1} \label{theorem: GoodwillieCathelineau}
Let $I$ be a nilpotent ideal in a commutative $\mathbb{Q}$-algebra $A$, the relative Chern character induces an isomorphism between the relative K-group $K_{n}(A,I)$ and negative cyclic homology $HN_{n}(A,I)$:
\[
  \mathrm{Ch}: K_{n}(A,I)_{\mathbb{Q}} \xrightarrow{\cong} HN_{n}(A,I).
\]
Furthermore, the relative Chern character respects Adams operations. That is,
\[
\mathrm{Ch}:  K_{n}^{(i)}(A,I)_{\mathbb{Q}}  \xrightarrow{\cong} HN_{n}^{(i)}(A,I),
\]
here $K_{n}^{(i)}$ and $HN_{n}^{(i)}$ are respective eigenspaces of $\psi^{k}=k^{i}$ and $\psi^{k}=k^{i+1}$.
\end{theorem}

In the Appendix B of \cite{CHW}, Corti$\tilde{n}$as-Haesemeyer-Weibel shows a space version of Goodwillie isomorphism. For every  nilpotent sheaf of ideal $I$, they define $K(O,I)$ and $HN(O,I)$ as the following presheaves respectively:
\[
 U \rightarrow K(O(U),I(U)),  U \rightarrow HN(O(U),I(U))
\]

They write $\mathcal{K}(O,I)$ and $\mathcal{HN}(O,I)$ for the presheaves of spectrum whose initial spaces are $K(O,I)$ and $HN(O,I)$ respectively. 
Moreover, they define $\mathcal{K}^{(i)}(O,I)$ as the homotopy fiber of $\mathcal{K}(O,I)$ on which $\psi^{k}-k^{i}$ acts acyclicly. $\mathcal{HN}^{(i)}(O,I)$ is defined similarly. 
Goodwillie and Cathelineau's isomorphisms can be generalized in the following way.
\begin{theorem} \cite{CHW} \label{theorem: CHW}
 The relative Chern character induces homotopy equivalence of spectra:
 \begin{equation}
\begin{cases}
 \begin{CD}
 \mathrm{Ch}: \mathcal{K}(O,I) \simeq \mathcal{HN}(O,I), \\
 \mathrm{Ch}: \mathcal{K}^{(i)}(O,I) \simeq \mathcal{HN}^{(i)}(O,I).
 \end{CD}
\end{cases}
\end{equation} 

\end{theorem}

Now, let $X$ be a noetherian scheme of finite type over a field $k$ of characteristic $0$. Let $Y$ be a closed subset in $X$ and $U = X - Y$ be the open complement. Let $\varepsilon$ be a nilpotent of order $n$: $\varepsilon ^{n} = 0$. 

Let $\mathbb{H}(X,\bullet)$ denote Thomason's hypercohomology of spectra. We have the following 
commutative diagrams(each column and row is a homotopy fibration):
\[
  \begin{CD}
     \mathbb{H}_{Y}(X, \mathcal{K}(O, \varepsilon)) @>>> \mathbb{H}(X, \mathcal{K}(O, \varepsilon)) @>>>  \mathbb{H}(U, \mathcal{K}(O, \varepsilon)) \\
     @VVV  @VVV   @VVV  \\
     \mathbb{H}_{Y}(X, \mathcal{K}(O_{X}[\varepsilon])) @>>> \mathbb{H}(X, \mathcal{K}(O_{X}[\varepsilon])) @>>>  \mathbb{H}(U, \mathcal{K}(O_{U}[\varepsilon])) \\
     @VVV  @VVV   @VVV  \\
     \mathbb{H}_{Y}(X, \mathcal{K}(O_{X})) @>>> \mathbb{H}(X, \mathcal{K}(O_{X})) @>>>  \mathbb{H}(U, \mathcal{K}(O_{U})) \\
  \end{CD}
\]
and
\[
  \begin{CD}
     \mathbb{H}_{Y}(X, \mathcal{HN}(O, \varepsilon)) @>>> \mathbb{H}(X, \mathcal{HN}(O, \varepsilon)) @>>>  \mathbb{H}(U, \mathcal{HN}(O, \varepsilon)) \\
     @VVV  @VVV   @VVV  \\
     \mathbb{H}_{Y}(X, \mathcal{HN}(O_{X}[\varepsilon])) @>>> \mathbb{H}(X, \mathcal{HN}(O_{X}[\varepsilon])) @>>>  \mathbb{H}(U, \mathcal{HN}(O_{U}[\varepsilon])) \\
     @VVV  @VVV   @VVV  \\
     \mathbb{H}_{Y}(X, \mathcal{HN}(O_{X})) @>>> \mathbb{H}(X, \mathcal{HN}(O_{X})) @>>>  \mathbb{H}(U, \mathcal{HN}(O_{U})). \\
  \end{CD}
\]

Since both $\mathcal{K}$ and $\mathcal{HN}$ satisfy Zariski descent, they satisfy
\[
\begin{cases}
 \begin{CD}
 \mathbb{H}_{Y}(X, \mathcal{K}(O_{X})) = \mathcal{K}(X \ \mathrm{on} \ Y), \ \mathbb{H}_{Y}(X, \mathcal{HN}(O_{X})) = \mathcal{HN}(X \ \mathrm{on} \ Y);  \\
 \mathbb{H}_{Y}(X, \mathcal{K}(O_{X}[\varepsilon])) = \mathcal{K}(X[\varepsilon] \ \mathrm{on} \ Y), \ \mathbb{H}_{Y}(X, \mathcal{HN}(O_{X}[\varepsilon])) = \mathcal{HN}(X[\varepsilon] \ \mathrm{on} \ Y).
 \end{CD}
\end{cases}
\]

 The above diagrams show the following result
\begin{corollary}  \label{corollary: hmtpFiber}
$\mathbb{H}_{Y}(X, \mathcal{K}(O, \varepsilon))$ and $\mathbb{H}_{Y}(X, \mathcal{HN}(O, \varepsilon))$ are the homotpy fibres of, respectively
\[
 \mathbb{H}_{Y}(X, \mathcal{K}(O_{X}[\varepsilon])) \rightarrow \mathbb{H}_{Y}(X, \mathcal{K}(O_{X})),  \mathbb{H}_{Y}(X, \mathcal{HN}(O_{X}[\varepsilon])) \rightarrow \mathbb{H}_{Y}(X, \mathcal{HN}(O_{X})).
\]
\end{corollary}

Let $K_{n}(X[\varepsilon] \ \mathrm{on} \ Y[\varepsilon], \varepsilon)$ and  $HN_{n}(X[\varepsilon] \ \mathrm{on} \ Y[\varepsilon], \varepsilon)$denote the relative groups with support, that is, the kernel of the natural projections:
\[
  K_{n}(X[\varepsilon] \ \mathrm{on} \ Y[\varepsilon]) \xrightarrow{\varepsilon =0} K_{n}(X \  \mathrm{on} \ Y),  HN_{n}(X[\varepsilon] \ \mathrm{on} \ Y[\varepsilon]) \xrightarrow{\varepsilon =0} HN_{n}(X \  \mathrm{on} \ Y).
\]
In other words, $K_{n}(X[\varepsilon] \ \mathrm{on} \ Y[\varepsilon], \varepsilon) = \mathbb{H}^{-n}_{Y}(X, \mathcal{K}(O, \varepsilon))$ and $HN_{n}(X[\varepsilon] \ \mathrm{on} \ Y[\varepsilon], \varepsilon) = \mathbb{H}^{-n}_{Y}(X, \mathcal{HN}(O, \varepsilon))$. So we have the following corollary of Theorem ~\ref{theorem: CHW}:
\begin{corollary} \label{corollary: GoodwillieWithSupport}
Let $X$ be a noetherian scheme of finite type over a field $k$ of characteristic $0$. Let $Y$ be a closed subset in $X$ and  let $\varepsilon$ be a nilpotent of order $n$: $\varepsilon ^{n} = 0$. 
We have
\[
 K_{n}(X[\varepsilon] \ \mathrm{on}  \ Y[\varepsilon], \varepsilon)_{\mathbb{Q}} = HN_{n}(X[\varepsilon] \ \mathrm{on}  \ Y[\varepsilon], \varepsilon)_{\mathbb{Q}}.
\]
\end{corollary}

According to the Appendix B of \cite{CHW}, there exists  the following two splitting fibrations
\begin{equation}
\begin{cases}
 \begin{CD}
  \mathcal{K}^{(i)}(O, \varepsilon) \rightarrow \mathcal{K}(O, \varepsilon) \rightarrow \prod_{j\neq i}\mathcal{K}^{(j)}(O, \varepsilon), \\
 \mathcal{HN}^{(i)}(O, \varepsilon) \rightarrow \mathcal{HN}(O, \varepsilon) \rightarrow \prod_{j\neq i}\mathcal{HN}^{(j)}(O, \varepsilon).
 \end{CD}
\end{cases}
\end{equation} 

Sine taking $\mathbb{H}_{Y}(X,-)$ perserves homotopy fibrations, there exists the following two splitting fibrations:
\begin{equation}
\begin{cases}
 \begin{CD}
 \mathbb{H}_{Y}(X, \mathcal{K}^{(i)}(O, \varepsilon)) \to \mathbb{H}_{Y}(X, \mathcal{K}(O, \varepsilon))  \xrightarrow{\psi^{k}-k^{i}}    \mathbb{H}_{Y}(X ,\prod_{j\neq i}\mathcal{K}^{(j)}(O, \varepsilon)), \\
 \mathbb{H}_{Y}(X, \mathcal{HN}^{(i)}(O, \varepsilon)) \to \mathbb{H}_{Y}(X, \mathcal{HN}(O, \varepsilon))  \xrightarrow{\psi^{k}-k^{i+1}}    \mathbb{H}_{Y}(X,\prod_{j\neq i}\mathcal{HN}^{(j)}(O, \varepsilon )).
 \end{CD}
\end{cases}
\end{equation} 

Passing to group level, we obtain the following results:
\begin{equation}
\begin{cases}
 \begin{CD}
 \mathbb{H}^{-n}_{Y}(X, \mathcal{K}^{(i)}(O, \varepsilon))_{\mathbb{Q}} = \{x \in \mathbb{H}^{-n}_{Y}(X, \mathcal{K}(O, \varepsilon))| \psi^{k}(x)-k^{i}x=0 \}, \\
 \mathbb{H}^{-n}_{Y}(X, \mathcal{HN}^{(i)}(O, \varepsilon))_{\mathbb{Q}} = \{x \in \mathbb{H}^{-n}_{Y}(X, \mathcal{HN}(O, \varepsilon))| \psi^{k}(x)-k^{i+1}x=0 \}.
 \end{CD}
\end{cases}
\end{equation} 

Let $K^{(i)}_{n}(X[\varepsilon] \ \mathrm{on} \ Y[\varepsilon], \varepsilon)_{\mathbb{Q}}$ and  $HN^{(i)}_{n}(X[\varepsilon] \ \mathrm{on} \ Y[\varepsilon], \varepsilon)_{\mathbb{Q}}$ denote the weight $i$ eigenspaces of relative groups with support, that is, the kernel of the natural projections:
\[
  K^{(i)}_{n}(X[\varepsilon] \ \mathrm{on} \ Y[\varepsilon])_{\mathbb{Q}} \xrightarrow{\varepsilon =0} K^{(i)}_{n}(X \  \mathrm{on} \ Y)_{\mathbb{Q}},  HN^{(i)}_{n}(X[\varepsilon] \ \mathrm{on} \ Y[\varepsilon])_{\mathbb{Q}} \xrightarrow{\varepsilon =0} HN^{(i)}_{n}(X \  \mathrm{on} \ Y)_{\mathbb{Q}}.
\]
Then we have the following identifications:
\begin{equation}
\begin{cases}
 \begin{CD}
 \mathbb{H}^{-n}_{Y}(X, \mathcal{K}^{(i)}(O, \varepsilon))_{\mathbb{Q}} = K^{(i)}_{n}(X[\varepsilon] \ \mathrm{on} \ Y[\varepsilon], \varepsilon)_{\mathbb{Q}}, \\
 \mathbb{H}^{-n}_{Y}(X, \mathcal{HN}^{(i)}(O, \varepsilon))_{\mathbb{Q}} = HN^{(i)}_{n}(X[\varepsilon] \ \mathrm{on} \ Y[\varepsilon], \varepsilon)_{\mathbb{Q}}.
 \end{CD}
\end{cases}
\end{equation} 

Therefore, the homotopy equivalences
\[ 
  \mathcal{K}(O, \varepsilon) \simeq \mathcal{HN}(O, \varepsilon),  \mathcal{K}^{(i)}(O, \varepsilon) \simeq\mathcal{HN}^{(i)}(O, \varepsilon)
\]
give us the following finer result:
\begin{theorem} \label{theorem: EigenGoodwillieSupport}
Let $X$ be a noetherian scheme of finite type over a field $k$ of characteristic $0$. Let $Y$ be a closed subset in $X$ and let $\varepsilon$ be a nilpotent of order $n$: $\varepsilon ^{n} = 0$, we have the following identification:
\[
 K_{n}^{(i)}(X[\varepsilon] \ \mathrm{on} \ Y[\varepsilon],\varepsilon)_{\mathbb{Q}} = HN_{n}^{(i)}(X[\varepsilon] \ \mathrm{on} \ Y[\varepsilon],\varepsilon)_{\mathbb{Q}}.
\]
\end{theorem}

Let $R$ be a regular noetherian domain, which is also a commutative $\mathbb{Q}$-algebra, and let $\varepsilon$ be a nilpotent of order $n$: $\varepsilon ^{n} = 0$. We consider $R[\varepsilon]$ as a graded $\mathbb{Q}$-algebra, it is known that the relative negative cyclic homology $HN_{m}^{(i)}(R[\varepsilon],\varepsilon)$ can be identified with absolute differentials as follows:
\begin{equation}  \label{corollary: EigenNegCycNegCyclic}
\begin{cases}
 HN_{m}^{(i)}(R[\varepsilon],\varepsilon) & = (\Omega^{{2i-m-1}}_{R/ \mathbb{Q}})^{ \oplus n-1}, \mathrm{for} \ \frac{m}{2} < i \leq m;\\
 HN_{m}^{(i)}(R[\varepsilon],\varepsilon) & = 0, \mathrm{else}.
\end{cases}
\end{equation} 

 \begin{equation}
 HN_{m}(R[\varepsilon],\varepsilon) = (\Omega^{{m-1}}_{R/ \mathbb{Q}}\oplus \Omega^{{m-3}}_{R/ \mathbb{Q}} \oplus \dots)^{\oplus n-1},
\end{equation}
the last term in the direct sum is $\Omega^{{1}}_{R/ \mathbb{Q}}$ or $R$, depending on $m$ even or odd.  E.g., Geller-Weibel \cite{GW} computed the eigenspace of K-groups of truncated polynomials, see Application 5.6 on page 29. See also the survey by Hesselholt \cite{He}.

 Now, we explicitly compute the relative negative cyclic groups with support.

\begin{theorem} \label{theorem: RelNegCycSupport}
Let $X$ be a smooth projective variety over a field $k$ of characteristic $0$ and $y \in X^{(p)}$.  Let $\varepsilon$ be a nilpotent of order $n$: $\varepsilon ^{n} = 0$. 
For any integer $m$,  we have 
\begin{equation}
\begin{cases}
 \begin{CD}
 HN_{m}(O_{X,y}[\varepsilon] \ \mathrm{on} \ y[\varepsilon],\varepsilon) \cong H_{y}^{p}(\Omega^{\bullet}_{X/\mathbb{Q}}), \\
  HN^{(i)}_{m}(O_{X,y}[\varepsilon] \ \mathrm{on} \ y[\varepsilon],\varepsilon) \cong H_{y}^{p}(\Omega^{\bullet,(i)}_{X/\mathbb{Q}}),
 \end{CD}
\end{cases}
\end{equation} 
where $(\Omega^{\bullet}_{X/\mathbb{Q}}=\Omega^{m+p-1}_{X/\mathbb{Q}}\oplus \Omega^{m+p-3}_{X/\mathbb{Q}}\oplus \dots)^{\oplus n-1}$ and 
\begin{equation}
\begin{cases}
 \Omega_{X/ \mathbb{Q}}^{\bullet,(i)} & = (\Omega^{{2i-(m+p)-1}}_{X/ \mathbb{Q}})^{\oplus n-1}, \mathrm{for} \  \frac{m+p}{2}  < \ i \leq m+p,\\
  \Omega_{X/ \mathbb{Q}}^{\bullet,(i)} & = 0, \mathrm{else}.
\end{cases}
\end{equation} 
\end{theorem}

\begin{proof}
$O_{X,y}$ is a regular local ring with dimension $p$, so the depth of $O_{X,y}$ is $p$. For each $n \in \mathbb{Z}$,  $\Omega^{n}_{O_{X,y}/\mathbb{Q}}$ can be written as a direct limit of 
direct sum of $O_{X,y}$'s(as $O_{X,y}$-module), though $\Omega^{n}_{O_{X,y}/\mathbb{Q}}$ is not of finite type. Therefore, $\Omega^{n}_{O_{X,y}/\mathbb{Q}}$ has depth $p$.

 $HN_{m}(O_{X,y}[\varepsilon] \ \mathrm{on} \ y[\varepsilon],\varepsilon)$ can be identified with the hypercohomology $\mathbb{H}_{y}^{-m}(O_{X,y},HN(O_{X,y}[\varepsilon],\varepsilon))$,
where $HN(O_{X,y}[\varepsilon],\varepsilon)$ is the relative negative cyclic complex, that is, the kernel of
\[
 HN(O_{X,y}[\varepsilon]) \xrightarrow{\varepsilon=0} HN(O_{X,y}).
\]

There is a spectral sequence :
\[
 H_{y}^{-m-q}(O_{X,y}, H^{q}(HN(O_{X,y}[\varepsilon],\varepsilon))) \Longrightarrow \mathbb{H}_{y}^{-m}(HN(O_{X,y}[\varepsilon],\varepsilon)).
\]

By Formula (3.8) above, we have
 \[
 H^{q}(HN(O_{X,y}[\varepsilon],\varepsilon))= HN_{-q}(O_{X,y}[\varepsilon],\varepsilon)= (\Omega^{-q-1}_{O_{X,y}/\mathbb{Q}}\oplus \Omega^{-q-3}_{O_{X,y}/\mathbb{Q}}\oplus \dots)^{\oplus n-1}.
 \]
 As each $\Omega^{n}_{O_{X,y}/\mathbb{Q}}$ has depth $p$, only $H_{y}^{p}(X,H^{q}(HN(O_{X,y}[\varepsilon],\varepsilon)))$ can survive because of the depth condition.
This means $q=-m-p$ and 
\[
 H^{-m-p}(HN(O_{X,y}[\varepsilon],\varepsilon))=HN_{m+p}(O_{X,y}[\varepsilon],\varepsilon)= (\Omega^{m+p-1}_{O_{X,y}/\mathbb{Q}}\oplus \Omega^{m+p-3}_{O_{X,y}/\mathbb{Q}}\oplus \dots)^{\oplus n-1}.
\]
Let's write 
\[
\Omega^{\bullet}_{X/\mathbb{Q}} = (\Omega^{m+p-1}_{X/\mathbb{Q}}\oplus \Omega^{m+p-3}_{X/\mathbb{Q}}\oplus \dots)^{\oplus n-1}, 
\]
thus 
\[
 \mathbb{H}_{y}^{-m}(HN(O_{X,y}[\varepsilon],\varepsilon))=H_{y}^{p}(\Omega^{\bullet}_{X/\mathbb{Q}}),
\]
this means
\[
 HN_{m}(O_{X,y}[\varepsilon] \ \mathrm{on} \ y[\varepsilon],\varepsilon)= H_{y}^{p}(\Omega^{\bullet}_{X/\mathbb{Q}}).
\]

The proof for $ HN_{m}^{(i)}$ works similarly.
\end{proof}

\begin{corollary} \label{corollary: OrderNRelVersion}
Under the same assumption as above, 
relative Chern character induces the following isomorphisms between relative K-groups and local cohomology groups:
\begin{equation}
\begin{cases}
 \begin{CD}
 K_{m}(O_{X,y}[\varepsilon] \ \mathrm{on} \ y[\varepsilon],\varepsilon)_{\mathbb{Q}} \cong H_{y}^{p}(\Omega^{\bullet}_{X/\mathbb{Q}}), \\
  K^{(i)}_{m}(O_{X,y}[\varepsilon] \ \mathrm{on} \ y[\varepsilon],\varepsilon)_{\mathbb{Q}} \cong H_{y}^{p}(\Omega^{\bullet,(i)}_{X/\mathbb{Q}}),
 \end{CD}
\end{cases}
\end{equation} 
where $\Omega^{\bullet}_{X/\mathbb{Q}}=(\Omega^{m+p-1}_{X/\mathbb{Q}}\oplus \Omega^{m+p-3}_{X/\mathbb{Q}}\oplus \dots)^{\oplus n-1}$; and
\begin{equation}
\begin{cases}
 \Omega_{X/ \mathbb{Q}}^{\bullet,(i)} & = (\Omega^{{2i-(m+p)-1}}_{X/ \mathbb{Q}})^{\oplus n-1}, \mathrm{for} \  \frac{m+p}{2}  < \ i \leq m+p;\\
  \Omega_{X/ \mathbb{Q}}^{\bullet,(i)} & = 0, \mathrm{else}.
\end{cases}
\end{equation} 

\end{corollary}

\subsection{Bloch's formula}
\label{Bloch's formula}
 In this subsection, $X$ is a smooth projective variety over a field $k$ of characteristic $0$ and 
 $X_{j}$ is the $j$-th  trivial infinitesimal thickening of $X$, that is, $X_{j}= X \times_{k} S_{j}$, 
where $S_{j}=\mathrm{Spec}(k[t]/(t^{j+1}))$. The aim of this subsection is to extend Bloch's formula from $X$ to its trivial infinitesimal thickening $X_j$.

Let $i_{j}: X \to X_{j+1}$ denote the closed immersion and $p_{j}: X_{j} \to X$ denote the natural projection. We have $p_{j} \circ  i_{j} = \mathrm{Identity}$. The morphism $i^{\ast}_{j}$ induces pull-back between K-theory spectra: $i^{\ast}_{j}: \mathcal{K}(X_{j}) \to \mathcal{K}(X)$. Furthermore it induces maps between coniveau spectral sequences, recalled in Section ~\ref{K-theoretic Chow groups of derived categoies}:
\[
i^{\ast}_{j}: E_{1}^{p,q}(X_{j}) \to E_{1}^{p,q}(X).
\]

Since the pull-back of $p_{j}$:
\[
p^{\ast}_{j}: E_{1}^{p,q}(X) \to E_{1}^{p,q}(X_{j}),
\]
splits 
\[
i^{\ast}_{j}: E_{1}^{p,q}(X_{j}) \to E_{1}^{p,q}(X),
\]
this gives us the following split commutative diagram:
{\footnotesize
\[
  \begin{CD}
     0 @. 0 @. 0\\
      @VVV @VVV @VVV\\
     \overline{K}_{q}(k(X)[t]/(t^{j+1})) @<<< K_{q}(k(X)[t]/(t^{j+1})) @>i^{\ast}_{j}>>  K_{q}(k(X))  \\
     @VVV @VVV @VVV\\
      \bigoplus\limits_{x \in X ^{(1)}} \overline{K}_{q-1}(O_{X_{j},x_{j}} \ \mathrm{on} \ x_{j}) @<<< \bigoplus\limits_{x_{j} \in X_{j} ^{(1)}}K_{q-1}(O_{X_{j},x_{j}} \ \mathrm{on} \ x_{j}) @>>>  \bigoplus\limits_{x \in X ^{(1)}}K_{q-1}(O_{X,x} \ \mathrm{on} \ x) \\
    @VVV  @VVV @VVV\\
     \dots @<<<  \dots @>>> \dots \\ 
      @VVV @VVV @VVV\\
     \bigoplus\limits_{x \in X ^{(q-1)}} \overline{K}_{1}(O_{X_{j},x_{j}} \ \mathrm{on} \ x_{j}) @<<<  \bigoplus\limits_{x_{j} \in X_{j} ^{(q-1)}}K_{1}(O_{X_{j},x_{j}} \ \mathrm{on} \ x_{j})
      @>>> \bigoplus\limits_{x \in X^{(q-1)}}K_{1}(O_{X,x} \ \mathrm{on} \ x) \\
     @V \overline{d}_{1,X_{j}}^{q-1,-q}VV @Vd_{1,X_{j}}^{q-1,-q}VV @Vd_{1,X}^{q-1,-q}VV\\
      \bigoplus\limits_{x \in X ^{(q)}} \overline{K}_{0}(O_{X_{j},x_{j}} \ \mathrm{on} \ x_{j}) @<<<  \bigoplus\limits_{x_{j} \in X_{j} ^{(q)}}K_{0}(O_{X_{j},x_{j}} \ \mathrm{on} \ x_{j})
      @>>> \bigoplus\limits_{x \in X ^{(q)}}K_{0}(O_{X,x} \ \mathrm{on} \ x) \\
     @V \overline{d}_{1,X_{j}}^{q,-q}VV  @Vd_{1,X_{j}}^{q,-q}VV @Vd_{1,X}^{q,-q}VV\\
     \bigoplus\limits_{x \in X ^{(q+1)}} \overline{K}_{-1}(O_{X_{j},x_{j}} \ \mathrm{on} \ x_{j}) @<<< \bigoplus\limits_{x_{j} \in X_{j} ^{(q+1)}}K_{-1}(O_{X_{j},x_{j}} \ \mathrm{on} \ x_{j})
      @>>> \bigoplus\limits_{x \in X ^{(q+1)}}K_{-1}(O_{X,x} \ \mathrm{on} \ x) \\
     @V\overline{d}_{1,X_{j}}^{q+1,-q}VV @Vd_{1,X_{j}}^{q+1,-q}VV @Vd_{1,X}^{q+1,-q}VV\\
     \dots @<<< \dots @>>> \dots \\ 
     @VVV @VVV @VVV\\
     \bigoplus\limits_{x \in X ^{(d)}} \overline{K}_{q-d}(O_{X_{j},x_{j}} \ \mathrm{on} \ x_{j}) @<<<  \bigoplus\limits_{x_{j} \in X_{j} ^{(d)}}K_{q-d}(O_{X_{j},x_{j}} \ \mathrm{on} \ x_{j}) @>>> \bigoplus\limits_{x \in X ^{(d)}}K_{q-d}(O_{X,x} \ \mathrm{on} \ x) \\
     @VVV @VVV @VVV\\
     0 @. 0 @. 0,
  \end{CD}
\]
}
where $\overline{K}_{\ast}(O_{X_{j},x_{j}} \ \mathrm{on} \ x_{j}) $ denotes the kernel of the natural projection
\[
K_{\ast}(O_{X_{j},x_{j}} \ \mathrm{on} \ x_{j}) \to  K_{\ast}(O_{X,x} \ \mathrm{on} \ x).
\]
And the differentials $\overline{d}_{1,X_{j}}^{p,q}$, which  is induced from $d_{1,X_{j}}^{p,q}$, satisfies 
$ d_{1,X_{j}}^{p,q} = d_{1,X}^{p,q} \oplus \overline{d}_{1,X_{j}}^{p,q}$.

Each term of the left column in the above diagram can be identified with local cohomology group, using Corollary ~\ref{corollary: OrderNRelVersion}.  In fact, the relative Chern character induces the following commutative diagram
\[
  \begin{CD}
     0 @. 0\\
      @VVV  @VVV\\
    \Omega_{k(X)/ \mathbb{Q}}^{\bullet} @<\cong<<  \overline{K}_{q}(k(X)[t]/(t^{j+1}))_{\mathbb{Q}}  \\
     @VVV @VVV\\
       \bigoplus \limits_{x \in X^{(1)}}H_{x}^{1}(\Omega_{X/\mathbb{Q}}^{\bullet})  @<<< \bigoplus\limits_{x \in X ^{(1)}} \overline{K}_{q-1}(O_{X_{j},x_{j}} \ \mathrm{on} \ x_{j})_{\mathbb{Q}}  \\
    @VVV   @VVV\\
     \dots @<<<  \dots \\ 
      @VVV  @VVV\\
     \bigoplus \limits_{x \in X^{(q-1)}}H_{x}^{q-1}(\Omega_{X/\mathbb{Q}}^{\bullet})  @<<< \bigoplus\limits_{x \in X ^{(q-1)}} \overline{K}_{1}(O_{X_{j},x_{j}} \ \mathrm{on} \ x_{j})_{\mathbb{Q}} \\
     @V\partial_{1}^{q-1,-q}VV @V \overline{d}_{1,X_{j}}^{q-1,-q}VV  \\
      \bigoplus \limits_{x \in X^{(q)}}H_{x}^{q}(\Omega_{X/\mathbb{Q}}^{\bullet}) @<<< \bigoplus\limits_{x \in X ^{(q)}} \overline{K}_{0}(O_{X_{j},x_{j}} \ \mathrm{on} \ x_{j})_{\mathbb{Q}} \\
     @V\partial_{1}^{q,-q}VV @V \overline{d}_{1,X_{j}}^{q,-q}VV  \\
      \bigoplus \limits_{x \in X^{(q+1)}}H_{x}^{q+1}(\Omega_{X/\mathbb{Q}}^{\bullet})  @<<< \bigoplus\limits_{x \in X ^{(q+1)}} \overline{K}_{-1}(O_{X_{j},x_{j}} \ \mathrm{on} \ x_{j})_{\mathbb{Q}} \\
     @V\partial_{1}^{q+1,-q}VV @V\overline{d}_{1,X_{j}}^{q+1,-q}VV  \\
     \dots @<<< \dots \\ 
     @VVV  @VVV\\
     \bigoplus \limits_{x \in X^{(d)}}H_{x}^{d}(\Omega_{X/\mathbb{Q}}^{\bullet})  @<<< \bigoplus\limits_{x \in X ^{(d)}} \overline{K}_{q-d}(O_{X_{j},x_{j}} \ \mathrm{on} \ x_{j})_{\mathbb{Q}} \\
     @VVV @VVV\\
     0 @. 0,
  \end{CD}
\]
where 
\begin{equation*}
\begin{cases}
 \Omega_{X/ \mathbb{Q}}^{\bullet} &=  \ (\Omega^{q-1}_{X/\mathbb{Q}}\oplus \Omega^{q-3}_{X/\mathbb{Q}}\oplus \dots)^{\oplus j},\\
 \Omega_{k(X)/ \mathbb{Q}}^{\bullet} &=  \ (\Omega^{q-1}_{k(X)/\mathbb{Q}}\oplus \Omega^{q-3}_{k(X)/\mathbb{Q}}\oplus \dots)^{\oplus j}. \
\end{cases}
\end{equation*} 
The left column is the classical Cousin complex \cite{H-1}.
Combining this diagram with the one of page 14, one has:
\begin{theorem} \label{theorem: theoremversion1}
Let $X$ be a smooth projective variety over a field $k$ of characteristic $0$ and 
 $X_{j}$ be the $j$-th  trivial infinitesimal thickening of $X$. For each integer $q \geq 1$, there exists the following commutative diagram in which the Zariski sheafification of each column is a flasque resolution of $ \Omega_{X/ \mathbb{Q}}^{\bullet}$,  $K_{q}(O_{X_{j}})_{\mathbb{Q}}$ and $K_{q}(O_{X})_{\mathbb{Q}}$ respectively. The left arrows are induced by Chern characters from K-theory to negative cyclic homology.
 {\footnotesize
\[
  \begin{CD}
     0 @. 0 @. 0\\
     @VVV  @VVV @VVV\\
     \Omega_{k(X)/ \mathbb{Q}}^{\bullet} @<\mathrm{Chern}<<  K_{q}(k(X)[t]/(t^{j+1}))_{\mathbb{Q}} @>i^{\ast}_{j}>> K_{q}(k(X))_{\mathbb{Q}}  \\
    @VVV  @VVV @VVV\\
      \bigoplus \limits_{x \in X^{(1)}}H_{x}^{1}(\Omega_{X/\mathbb{Q}}^{\bullet})  @<<< \bigoplus\limits_{x_{j} \in X_{j} ^{(1)}}K_{q-1}(O_{X_{j},x_{j}} \ \mathrm{on} \ x_{j})_{\mathbb{Q}} @>>>  \bigoplus\limits_{x \in X ^{(1)}}K_{q-1}(O_{X,x} \ \mathrm{on} \ x)_{\mathbb{Q}} \\
     @VVV @VVV @VVV\\
     \dots @<<< \dots @>>> \dots \\ 
      @VVV @VVV @VVV\\
     \bigoplus \limits_{x \in X^{(q-1)}}H_{x}^{q-1}(\Omega_{X/\mathbb{Q}}^{\bullet})  @<<< \bigoplus\limits_{x_{j} \in X_{j} ^{(q-1)}}K_{1}(O_{X_{j},x_{j}} \ \mathrm{on} \ x_{j})_{\mathbb{Q}}
      @>>> \bigoplus\limits_{x \in X^{(q-1)}}K_{1}(O_{X,x} \ \mathrm{on} \ x)_{\mathbb{Q}} \\
     @V\partial_{1}^{q-1,-q}VV @Vd_{1,X_{j}}^{q-1,-q}VV @Vd_{1,X}^{q-1,-q}VV\\
     \bigoplus \limits_{x \in X^{(q)}}H_{x}^{q}(\Omega_{X/\mathbb{Q}}^{\bullet})  @<<< \bigoplus\limits_{x_{j} \in X_{j} ^{(q)}}K_{0}(O_{X_{j},x_{j}} \ \mathrm{on} \ x_{j})_{\mathbb{Q}}
      @>>> \bigoplus\limits_{x \in X ^{(q)}}K_{0}(O_{X,x} \ \mathrm{on} \ x)_{\mathbb{Q}} \\
     @V\partial_{1}^{q,-q}VV @Vd_{1,X_{j}}^{q,-q}VV @Vd_{1,X}^{q,-q}VV\\
     \bigoplus \limits_{x \in X^{(q+1)}}H_{x}^{q+1}(\Omega_{X/\mathbb{Q}}^{\bullet})  @<<< \bigoplus\limits_{x_{j} \in X_{j} ^{(q+1)}}K_{-1}(O_{X_{j},x_{j}} \ \mathrm{on} \ x_{j})_{\mathbb{Q}}
      @>>> \bigoplus\limits_{x \in X ^{(q+1)}}K_{-1}(O_{X,x} \ \mathrm{on} \ x)_{\mathbb{Q}} \\
     @V\partial_{1}^{q+1,-q}VV @Vd_{1,X_{j}}^{q+1,-q}VV @Vd_{1,X}^{q+1,-q}VV\\
    \dots @<<< \dots @>>> \dots \\ 
     @VVV @VVV @VVV\\
    \bigoplus \limits_{x \in X^{(d)}}H_{x}^{d}(\Omega_{X/\mathbb{Q}}^{\bullet}) @<<<  \bigoplus\limits_{x_{j} \in X_{j} ^{(d)}}K_{q-d}(O_{X_{j},x_{j}} \ \mathrm{on} \ x_{j})_{\mathbb{Q}} @>>> \bigoplus\limits_{x \in X ^{(d)}}K_{q-d}(O_{X,x} \ \mathrm{on} \ x) _{\mathbb{Q}} \\
     @VVV @VVV @VVV\\
      0 @. 0 @. 0,
  \end{CD}
\]
}
where 
\begin{equation*}
\begin{cases}
 \Omega_{X/ \mathbb{Q}}^{\bullet} &=  \ (\Omega^{q-1}_{X/\mathbb{Q}}\oplus \Omega^{q-3}_{X/\mathbb{Q}}\oplus \dots)^{\oplus j},\\
 \Omega_{k(X)/ \mathbb{Q}}^{\bullet} &=  \ (\Omega^{q-1}_{k(X)/\mathbb{Q}}\oplus \Omega^{q-3}_{k(X)/\mathbb{Q}}\oplus \dots)^{\oplus j}. \
\end{cases}
\end{equation*}

\end{theorem}

\begin{proof}
Since $X$ is smooth, the sheafifications of the left and right column are flasque resolutions of $ \Omega_{X/ \mathbb{Q}}^{\bullet}$ and $K_{q}(O_{X})_{\mathbb{Q}}$ respectively. So the sheafification of the middle column is a flasque resolution of $K_{q}(O_{X_{j}})_{\mathbb{Q}}$.
\end{proof}
 
Next, we want to use Adams operations $\psi^{k}$ to refine the diagram in Theorem ~\ref{theorem: theoremversion1}. Since $q$ can be any integer($\geq 1$) there, negative K-groups might appear. One can use Weibel's method to extend Adams operations to negative K-groups, as recalled in the beginning of Section ~\ref{Definition}. In the following, $K^{(i)}_{\ast}$ denotes the eigenspace of $\psi^{k}=k^{i}$.

Since $X$ is smooth, we note that negative K-groups of the right column Theorem ~ \ref{theorem: theoremversion1} are 0 and this Gersten complex agrees with the one in \cite{Q},
{\footnotesize
\begin{align*}
 0 \to \bigoplus_{x \in X^{(0)}}K_{q}(k(X)) \to \dots \to \bigoplus_{x \in X^{(q-1)}}K_{1}(k(x)) 
  \xrightarrow{d_{1,X}^{q-1,-q}}  \bigoplus_{x \in X^{(q)}}K_{0}(k(x)) \xrightarrow{d_{1,X}^{q,-q}} 0.
\end{align*}
}

Thus, we can use Adams operations , defined at space level \cite{S}, to decompose this complex directly. 
And for every $i \in \mathbb{Z}$, then there exists the following finer complex  
{\footnotesize
\begin{align*}
 0  \to K_{q}^{(i)}(k(X))_{\mathbb{Q}}  \to \dots \to \bigoplus_{x \in X^{(q-1)}}K_{1}^{(i)}(O_{X,x} \ \mathrm{on} \ x)_{\mathbb{Q}} 
   \xrightarrow{d_{1,X}^{q-1,-q}} \bigoplus_{x \in X^{(q)}}K_{0}^{(i)}(O_{X,x} \ \mathrm{on} \ x)_{\mathbb{Q}}  \xrightarrow{d_{1,X}^{q,-q}} 0, 
\end{align*}
}
which agrees with the following complex in  \cite{S},
{\footnotesize
\begin{align*}
0  \to K_{q}^{(i)}(k(X))_{\mathbb{Q}}  \to \dots \to \bigoplus_{x \in X^{(q-1)}}K_{1}^{(i-(q-1))}(k(x))_{\mathbb{Q}}  
   \xrightarrow{d_{1,X}^{q-1,-q}} \bigoplus_{x \in X^{(q)}}K_{0}^{(i-q)}(k(x))_{\mathbb{Q}}  \xrightarrow{d_{1,X}^{q,-q}} 0.
\end{align*}
}

 One notes the differentials $\partial_{1}^{p,q}$ of the Cousin complex(the left column in Theorem ~\ref{theorem: theoremversion1}) respects Adams operations(Adams operations on $H_{x}^{\ast}(\Omega_{X/\mathbb{Q}}^{\bullet})$ is induced from the isomorphism in Theorem ~\ref{theorem: RelNegCycSupport}):
\[
\partial_{1}^{p,q}:  \bigoplus \limits_{x \in X^{(p)}}H_{x}^{p}(\Omega_{X/\mathbb{Q}}^{\bullet, (i)}) \to \bigoplus \limits_{x \in X^{(p+1)}}H_{x}^{p+1}(\Omega_{X/\mathbb{Q}}^{\bullet,(i)}), 
\]
where 
\begin{equation}
\begin{cases}
 \Omega_{X/ \mathbb{Q}}^{\bullet,(i)} & =  \ (\Omega^{{2i-q-1}}_{X/ \mathbb{Q}})^{ \oplus j},   \mathrm{for} \  \frac{q}{2}  < \ i \leq q;\\
  \Omega_{X/ \mathbb{Q}}^{\bullet,(i)} &=  \  0, \mathrm{else}. \
\end{cases}
\end{equation}

Since the relative Chern character respects Adams operations, see Theorem ~\ref{theorem: CHW} and Corollary ~\ref{corollary: OrderNRelVersion}, there exists the following commutative diagram
\[
  \begin{CD}
     0 @. 0\\
      @VVV  @VVV\\
    \Omega_{k(X)/ \mathbb{Q}}^{\bullet,(i)} @<\cong<<  \overline{K}^{(i)}_{q}(k(X)[t]/(t^{j+1}))_{\mathbb{Q}}  \\
     @VVV @VVV\\
       \bigoplus \limits_{x \in X^{(1)}}H_{x}^{1}(\Omega_{X/\mathbb{Q}}^{\bullet,(i)})  @<<< \bigoplus\limits_{x \in X ^{(1)}} \overline{K}^{(i)}_{q-1}(O_{X_{j},x_{j}} \ \mathrm{on} \ x_{j})_{\mathbb{Q}}  \\
    @VVV   @VVV\\
     \dots @<<<  \dots \\ 
      @VVV  @VVV\\
     \bigoplus \limits_{x \in X^{(q-1)}}H_{x}^{q-1}(\Omega_{X/\mathbb{Q}}^{\bullet,(i)})  @<<< \bigoplus\limits_{x \in X ^{(q-1)}} \overline{K}^{(i)}_{1}(O_{X_{j},x_{j}} \ \mathrm{on} \ x_{j})_{\mathbb{Q}} \\
     @V\partial_{1}^{q-1,-q}VV @V \overline{d}_{1,X_{j}}^{q-1,-q}VV  \\
      \bigoplus \limits_{x \in X^{(q)}}H_{x}^{q}(\Omega_{X/\mathbb{Q}}^{\bullet,(i)}) @<<< \bigoplus\limits_{x \in X ^{(q)}} \overline{K}^{(i)}_{0}(O_{X_{j},x_{j}} \ \mathrm{on} \ x_{j})_{\mathbb{Q}} \\
     @V\partial_{1}^{q,-q}VV @V \overline{d}_{1,X_{j}}^{q,-q}VV  \\
      \bigoplus \limits_{x \in X^{(q+1)}}H_{x}^{q+1}(\Omega_{X/\mathbb{Q}}^{\bullet,(i)})  @<<< \bigoplus\limits_{x \in X ^{(q+1)}} \overline{K}^{(i)}_{-1}(O_{X_{j},x_{j}} \ \mathrm{on} \ x_{j})_{\mathbb{Q}} \\
     @V\partial_{1}^{q+1,-q}VV @V\overline{d}_{1,X_{j}}^{q+1,-q}VV  \\
     \dots @<<< \dots \\ 
     @VVV  @VVV\\
     \bigoplus \limits_{x \in X^{(d)}}H_{x}^{d}(\Omega_{X/\mathbb{Q}}^{\bullet,(i)})  @<<< \bigoplus\limits_{x \in X ^{(d)}} \overline{K}^{(i)}_{q-d}(O_{X_{j},x_{j}} \ \mathrm{on} \ x_{j})_{\mathbb{Q}} \\
     @VVV @VVV\\
     0 @. 0,
  \end{CD}
\]
where $\overline{K}^{(i)}_{\ast}(O_{X_{j},x_{j}} \ \mathrm{on} \ x_{j})$ denotes the eigenspace of $\psi^{k}=k^{i}$.
Consequently, $\overline{d}_{1,X_{j}}^{p,q}$ respects Adams operations.
The differential $d_{1,X_{j}}^{p,q}$ splits as $d_{1,X}^{p,q} \oplus \overline{d}_{1,X_{j}}^{p,q}$, so $d_{1,X_{j}}^{p,q}$ respects Adams operations:
\[
d_{1,X_{j}}^{p,q}: E_{1}^{p,q}(X_{j})^{(i)} \to E_{1}^{p,q}(X_{j})^{(i)}.
\]
So we have the following complex:
{\footnotesize
\begin{align*}
 0 \to & K_{q}^{(i)}(X_{j})_{\mathbb{Q}}   \to K_{q}^{(i)}(k(X_{j}))_{\mathbb{Q}}  \to \dots \to \bigoplus_{x_{j} \in X_{j}^{(q-1)}}K_{1}^{(i)}(O_{X_{j},x_{j}} \ \mathrm{on} \ x_{j})_{\mathbb{Q}}  
   \xrightarrow{d_{1,X_{j}}^{q-1,-q}} \bigoplus_{x_{j} \in X_{j}^{(q)}}K_{0}^{(i)}(O_{X_{j},x_{j}} \ \mathrm{on} \ x_{j})_{\mathbb{Q}}   \\  & \xrightarrow{d_{1,X_{j}}^{q,-q}} \bigoplus_{x_{j} \in X_{j}^{(q-1)}}K_{-1}^{(i)}(O_{X_{j},x_{j}} \ \mathrm{on} \ x_{j})_{\mathbb{Q}}   \to \dots \to \bigoplus_{x_{j} \in X_{j}^{(d)}}K_{q-d}^{(i)}(O_{X_{j},x_{j}} \ \mathrm{on} \ x_{j})_{\mathbb{Q}}  \to 0.
\end{align*}
}

\newpage

\begin{theorem} \label{theorem: Adamsdiagram}
Let $X$ be a smooth projective variety over a field $k$ of characteristic $0$ and 
 $X_{j}$ be the $j$-th  trivial infinitesimal thickening of $X$. For each integer $q \geq 1$, there exists the following commutative diagram in which the Zariski sheafification of each column is a flasque resolution of $\Omega_{X/ \mathbb{Q}}^{\bullet,(i)}$,  $K^{(i)}_{q}(O_{X_{j}})_{\mathbb{Q}} $ and $K^{(i)}_{q}(O_{X})_{\mathbb{Q}}$ respectively. The left arrows are induced by Chern characters from K-theory to negative cyclic homology.
 
{\footnotesize
\[
  \begin{CD}
     0 @. 0 @. 0\\
     @VVV @VVV @VVV\\
     \Omega_{k(X)/ \mathbb{Q}}^{\bullet,(i)} @<\mathrm{Chern}<<  K^{(i)}_{q}(k(X)[t]/(t^{j+1}))_{\mathbb{Q}}  @>i^{\ast}_{j}>> K^{(i)}_{q}(k(X))_{\mathbb{Q}}  \\
     @VVV @VVV @VVV\\
     \bigoplus \limits_{x \in X^{(1)}}H_{x}^{1}(\Omega_{X/\mathbb{Q}}^{\bullet,(i)}) @<<< \bigoplus \limits_{x_{j} \in X_{j}^{(1)}}K^{(i)}_{q-1}(O_{X_{j},x_{j}} \ \mathrm{on} \ x_{j})_{\mathbb{Q}}  @>>>  \bigoplus \limits_{x \in X^{(1)}}K^{(i)}_{q-1}(O_{X,x} \ \mathrm{on} \ x)_{\mathbb{Q}}  \\
     @VVV @VVV @VVV\\
      \dots @<<< \dots @>>> \dots \\ 
     @VVV @VVV @VVV\\
     \bigoplus\limits_{x \in X^{(q-1)}}H_{x}^{q-1}(\Omega_{X/ \mathbb{Q}}^{\bullet,(i)}) @<<< \bigoplus \limits_{x_{j} \in X_{j}^{(q-1)}}K^{(i)}_{1}(O_{X_{j},x_{j}} \ \mathrm{on} \ x_{j})_{\mathbb{Q}} 
    @>>> \bigoplus \limits_{x \in X^{(q-1)}}K^{(i)}_{1}(O_{X,x} \ \mathrm{on} \ x)_{\mathbb{Q}}  \\
     @VVV @Vd_{1, X_{j}}^{q-1,-q}VV @Vd_{1,X}^{q-1,-q}VV\\
     \bigoplus \limits_{x \in X^{(q)}}H_{x}^{q}(\Omega_{X/ \mathbb{Q}}^{\bullet,(i)}) @<<< \bigoplus \limits_{x_{j} \in X_{j}^{(q)}}K^{(i)}_{0}(O_{X_{j},x_{j}} \ \mathrm{on} \ x_{j})_{\mathbb{Q}} 
     @>>> \bigoplus\limits_{x \in X^{(q)}}K^{(i)}_{0}(O_{X,x} \ \mathrm{on} \ x)_{\mathbb{Q}}  \\
     @VVV @Vd_{1,X_{j}}^{q,-q}VV @Vd_{1,X}^{q,-q}VV\\
     \bigoplus\limits_{x \in X^{(q+1)}}H_{x}^{q+1}(\Omega_{X/ \mathbb{Q}}^{\bullet,(i)}) @<<< \bigoplus \limits_{x_{j} \in X_{j}^{(q+1)}}K^{(i)}_{-1}(O_{X_{j},x_{j}} \ \mathrm{on} \ x_{j})_{\mathbb{Q}} 
      @>>> \bigoplus\limits_{x \in X^{(q+1)}}K^{(i)}_{-1}(O_{X,x} \ \mathrm{on} \ x)_{\mathbb{Q}}  \\
     @VVV @VVV @VVV\\
     \dots @<<< \dots @>>> \dots \\ 
     @VVV @VVV @VVV\\
     \bigoplus \limits_{x\in X^{(d)}}H_{x}^{d}(\Omega_{X/ \mathbb{Q}}^{\bullet,(i)}) @<<< \bigoplus \limits_{x_{j} \in X_{j}^{(d)}}K^{(i)}_{q-d}(O_{X_{j},x_{j}} \ \mathrm{on} \ x_{j})_{\mathbb{Q}} 
    @>>>  \bigoplus\limits_{x \in X^{(d)}}K^{(i)}_{q-d}(O_{X,x} \ \mathrm{on} \ x)_{\mathbb{Q}}  \\
     @VVV @VVV @VVV\\
      0 @. 0 @. 0,
  \end{CD}
\]
}
where 
\begin{equation}
\begin{cases}
 \Omega_{X/ \mathbb{Q}}^{\bullet,(i)} & =  \ (\Omega^{{2i-q-1}}_{X/ \mathbb{Q}})^{ \oplus j},   \mathrm{for} \  \frac{q}{2}  < \ i \leq q;\\
  \Omega_{X/ \mathbb{Q}}^{\bullet,(i)} &=  \  0, \mathrm{else}. \
\end{cases}
\end{equation} 

\end{theorem}

\begin{proof}
The sheafifications of the left and right columns are flasque resolutions of $\Omega_{X/ \mathbb{Q}}^{\bullet,(i)}$,  and $K^{(i)}_{q}(O_{X})_{\mathbb{Q}}$ respectively. 
So the sheafification of the middle column is a flasque resolution of $K^{(i)}_{q}(O_{X_{j}})_{\mathbb{Q}}$.
 
\end{proof}

\newpage

In particular, we are interested in  the ``Milnor K-theory''. Letting $i=q$ in Theorem ~\ref{theorem: Adamsdiagram}, one have the following theorem.

\begin{theorem} \label{theorem: Milnorversion}
Let $X$ be a smooth projective variety over a field $k$ of characteristic $0$ and 
 $X_{j}$ be the $j$-th  trivial infinitesimal thickening of $X$. For each integer $q \geq 1$, there exists the following commutative diagram in which the Zariski sheafification of each column is a flasque resolution of $(\Omega_{X/ \mathbb{Q}}^{q-1})^{\oplus j}$,  $K^{M}_{q}(O_{X_{j}})_{\mathbb{Q}}$ and $K^{M}_{q}(O_{X})_{\mathbb{Q}}$ respectively. The left arrows are induced by Chern characters from K-theory to negative cyclic homology.

{\footnotesize
\[
  \begin{CD}
     0 @. 0 @. 0\\
     @VVV @VVV @VVV\\
     (\Omega_{k(X)/ \mathbb{Q}}^{q-1})^{\oplus j} @<\mathrm{Chern}<<  K^{M}_{q}(k(X)[t]/(t^{j+1})) @>i^{\ast}_{j}>> K^{M}_{q}(k(X)) \\
     @VVV @VVV @VVV\\
     \bigoplus \limits_{x \in X^{(1)}}H_{x}^{1}((\Omega_{X/ \mathbb{Q}}^{q-1})^{\oplus j}) @<<< \bigoplus \limits_{x_{j} \in X_{j}^{(1)}}K^{M}_{q-1}(O_{X_{j},x_{j}} \ \mathrm{on} \ x_{j}) @>>>  \bigoplus \limits_{x \in X^{(1)}}K^{M}_{q-1}(O_{X,x} \ \mathrm{on} \ x)\\
     @VVV @VVV @VVV\\
      \dots @<<< \dots @>>> \dots \\ 
     @VVV @VVV @VVV\\
     \bigoplus\limits_{x \in X^{(q-1)}}H_{x}^{q-1}((\Omega_{X/ \mathbb{Q}}^{q-1})^{\oplus j}) @<<< \bigoplus \limits_{x_{j} \in X_{j}^{(q-1)}}K^{M}_{1}(O_{X_{j},x_{j}} \ \mathrm{on} \ x_{j})
    @>>> \bigoplus \limits_{x \in X^{(q-1)}}K^{M}_{1}(O_{X,x} \ \mathrm{on} \ x) \\
     @VVV @Vd_{1, X_{j}}^{q-1,-q}VV @Vd_{1,X}^{q-1,-q}VV\\
     \bigoplus \limits_{x \in X^{(q)}}H_{x}^{q}((\Omega_{X/ \mathbb{Q}}^{q-1})^{\oplus j}) @<<< \bigoplus \limits_{x_{j} \in X_{j}^{(q)}}K^{M}_{0}(O_{X_{j},x_{j}} \ \mathrm{on} \ x_{j})
     @>>> \bigoplus\limits_{x \in X^{(q)}}K^{M}_{0}(O_{X,x} \ \mathrm{on} \ x) \\
     @VVV @Vd_{1,X_{j}}^{q,-q}VV @Vd_{1,X}^{q,-q}VV\\
     \bigoplus\limits_{x \in X^{(q+1)}}H_{x}^{q+1}((\Omega_{X/ \mathbb{Q}}^{q-1})^{\oplus j}) @<<< \bigoplus \limits_{x_{j} \in X_{j}^{(q+1)}}K^{M}_{-1}(O_{X_{j},x_{j}} \ \mathrm{on} \ x_{j})
      @>>> \bigoplus\limits_{x \in X^{(q+1)}}K^{M}_{-1}(O_{X,x} \ \mathrm{on} \ x) \\
     @VVV @VVV @VVV\\
     \dots @<<< \dots @>>> \dots \\ 
     @VVV @VVV @VVV\\
     \bigoplus \limits_{x\in X^{(d)}}H_{x}^{d}((\Omega_{X/ \mathbb{Q}}^{q-1})^{\oplus j}) @<<< \bigoplus \limits_{x_{j} \in X_{j}^{(d)}}K^{M}_{q-d}(O_{X_{j},x_{j}} \ \mathrm{on} \ x_{j})
    @>>>  \bigoplus\limits_{x \in X^{(d)}}K^{M}_{q-d}(O_{X,x} \ \mathrm{on} \ x) \\
     @VVV @VVV @VVV\\
      0 @. 0 @. 0.
  \end{CD}
\]
}
\end{theorem}

The middle and right columns are complexes, so Definition ~\ref{definition: MilnorKChow} applies:
\begin{definition} \label{corollary: MilnorVerWellDef}
The Milnor K-theoretic $q$-cycles and rational equivalence of $(X,O_{X})$ are defined to be 
\[
 Z^{M}_{q}(D^{\mathrm{perf}}(X)) = \mathrm{Ker}(d_{1,X}^{q,-q}), \  Z_{q,\mathrm{rat}}^{M}(D^{\mathrm{perf}}(X)) =\mathrm{Im}(d_{1,X}^{q-1,-q}).
\]

The $q$-th Milnor K-theoretic Chow group  of $(X,O_{X})$ is defined to be
\[
 CH^{M}_{q}(D^{\mathrm{perf}}(X))= \dfrac{Z^{M}_{q}(D^{\mathrm{perf}}(X))}{Z_{q,\mathrm{rat}}^{M}(D^{\mathrm{perf}}(X))}.
\]

The Milnor K-theoretic $q$-cycles and rational equivalence of $(X,O_{X_{j}})$ are defined to be
\[
 Z^{M}_{q}(D^{\mathrm{perf}}(X_{j}))= \mathrm{Ker}(d_{1,X_{j}}^{q,-q}), \ Z_{q,\mathrm{rat}}^{M}(D^{\mathrm{perf}}(X_{j}))=\mathrm{Im}(d_{1,X_{j}}^{q-1,-q}).
\]

The $q$-th Milnor K-theoretic Chow group  of $(X,O_{X_{j}})$ is defined to be
\[
 CH^{M}_{q}(D^{\mathrm{perf}}(X_{j}))= \dfrac{Z^{M}_{q}(D^{\mathrm{perf}}(X_{j}))}{Z_{q,\mathrm{rat}}^{M}(D^{\mathrm{perf}}(X_{j}))}.
\]

\end{definition}

\textbf{Agreement}. We now prove that our  Milnor K-theoretic Chow group rationally agrees with the classical ones for smooth projective varieties.
\begin{theorem} 
For $X$ a smooth projective variety over a field $k$ of characteristic $0$, let $Z^{q}(X)$, $Z_{\mathrm{rat}}^{q}(X)$ and $CH^{q}(X)$ denote the classical $q$-cycles, rational equivalence and Chow groups respectively, then we have the following identifications
\[
 Z_{q}^{M}(D^{\mathrm{perf}}(X))= Z^{q}(X)_{\mathbb{Q}},
\]
\[
  Z_{q,\mathrm{rat}}^{M}(D^{\mathrm{perf}}(X)) = Z_{\mathrm{rat}}^{q}(X)_{\mathbb{Q}},
\]
\[
 CH^{M}_{q}(D^{\mathrm{perf}}(X)) = CH^{q}(X)_{\mathbb{Q}}.
\]
\end{theorem}

\begin{proof}
Since $X$ is smooth, by Theorem ~\ref{theorem: Soule} and
Theorem ~\ref{theorem: SouleRRWithoutdenominator}, the right column in Theorem ~\ref{theorem: Milnorversion} agrees with the following classical sequence,
\begin{footnotesize}
\begin{align*}
 0 \to &  K^{M}_{q}(k(X))_{\mathbb{Q}} \to \bigoplus_{x \in X^{(1)}}K^{M}_{q-1}(k(x))_{\mathbb{Q}} \to \dots  \to \bigoplus_{x \in X^{(q-1)}}K^{M}_{1}(k(x))_{\mathbb{Q}} \\
 &  \xrightarrow{d_{1,X}^{q-1,-q}} \bigoplus_{x \in X^{(q)}}K^{M}_{0}(k(x))_{\mathbb{Q}} \xrightarrow{d_{1,X}^{q,-q}} 0.
\end{align*}
\end{footnotesize}

Noting $K^{M}_{0}(k(x))_{\mathbb{Q}}= \mathbb{Q}$, one has
\[
  Z^{M}_{q}(D^{\mathrm{perf}}(X))= \mathrm{Ker}(d_{1,X}^{q,-q})= \bigoplus_{x \in X^{(q)}}K^{M}_{0}(k(x))_{\mathbb{Q}} = Z^{q}(X)_{\mathbb{Q}}.
\]

Since $K^{M}_{1}(k(x))_{\mathbb{Q}}= K_{1}(k(x))_{\mathbb{Q}}$, as explained in Quillen's proof of Bloch's formula, the image of $d_{1,X}^{q-1,-q}$ gives the rational equivalence. Hence, 
\[
 Z^{M}_{q,\mathrm{rat}}(D^{\mathrm{perf}}(X)) = \mathrm{Im}(d_{1,X}^{q-1,-q}) = Z^{q}_{\mathrm{rat}}(X)_{\mathbb{Q}}.
\]

Therefore, we have the following identification
\[
  CH^{M}_{q}(D^{\mathrm{perf}}(X)) = CH^{q}(X)_{\mathbb{Q}}.
\]
\end{proof}

We obtain the following \textbf{Bloch's formulas}, which gives a positive answer to Green-Griffiths' \textbf{Question 1.1} for trivial deformations.

\begin{theorem} \label{theorem: ExtendBlochFormula}
Let $X$ be a smooth projective variety over a field $k$ of characteristic $0$ and $X_{j}$ be the $j$-th trivial infinitesimal deformation of $X$. For $q$ a non negative integer, 
we have the following identifications:
\[ 
 CH^{M}_{q}(D^{\mathrm{perf}}(X))  = H^{q}(X,K_{q}^{M}(O_{X}))_{\mathbb{Q}}.
\]
 \[
  CH^{M}_{q}(D^{\mathrm{perf}}(X_{j})) = H^{q}(X,K_{q}^{M}(O_{X_{j}}))_{\mathbb{Q}}.
 \]
\end{theorem}

\begin{proof}
Immediately from Theorem ~\ref{theorem: Milnorversion}.
\end{proof}

\end{document}